\documentclass[pdftex,a4paper,12pt]{amsart}

\usepackage{geometry}

\usepackage[cp1250]{inputenc}

\usepackage{alltt}
\usepackage{euler}
\usepackage{amsfonts}
\usepackage{amscd}
\usepackage{graphicx}
\usepackage{lpic}
\usepackage{longtable}
\usepackage{url}
\usepackage{import}
\usepackage{amssymb, amsthm, amsmath}
\usepackage[sc]{mathpazo}
\usepackage{enumitem, textcomp, setspace}
\usepackage{nicefrac}
\usepackage{varwidth}

\usepackage{tikz}
\usepackage{tikz-cd}
\usetikzlibrary{bbox}
\usetikzlibrary{intersections}
\usepackage{pgfplots}
\usepgfplotslibrary{external}

\allowdisplaybreaks

\numberwithin{equation}{section}

\theoremstyle{plain}
\newtheorem{theorem}{Theorem}[section]

\theoremstyle{definition}

\newtheorem{remark}[theorem]{Remark}

\usepackage[colorlinks, pagebackref=false]{hyperref}

\usepackage[all]{hypcap}

\definecolor{internalLink}{rgb}{0,0,0.5}
\definecolor{citeLink}{rgb}{0,0.5,0}
\definecolor{urlLink}{rgb}{0,0.5,0.5}

\hypersetup{
	linkcolor=internalLink,
	citecolor=citeLink,
	urlcolor=urlLink
}

\title[A polynomial pair invariant of alternating knots and links]{A polynomial pair invariant of alternating knots\\ and links}

\author{Micha{\l} Jab{\l}onowski}

\address{Institute of Mathematics, Faculty of Mathematics, Physics and Informatics,\newline University of Gda\'nsk, 80-308 Gda\'nsk, Poland}

\keywords{knot invariant, link invariant, alternating knot, WRP invariant, knots tabulation}

\subjclass[2020]{57K10 (primary), 57K14 (secondary)}

\email{michal.jablonowski@gmail.com}

\date{\today}

\begin{document}

\maketitle

\begin{abstract}
	
We introduce an invariant of alternating knots and links (called here WRP), namely a pair of integer polynomials associated with their two checkerboard planar graphs from their minimal diagram. We prove that the invariant is well-defined and give its values obtained from calculations for some knots in the tables. This invariant is strong enough to distinguish all knots in the tables with up to 10 crossings (including their mirror images). We compare the strength of the new invariant with classical invariants, including the three-variable Kauffman bracket.

\end{abstract}
\vspace{-0.3cm}
\section{Introduction}

We introduce an invariant of alternating knots and links, called  $WRP$, namely a set of an unordered pair of two variable ($w, r$) integer polynomials associated with their two checkerboard planar graphs from their reduced diagram. We prove that the invariant is well-defined and give its values obtained from calculations for some knots in the tables. This invariant is strong enough to distinguish all knots in the tables with up to $10$ crossings (including their mirror images in the sense of the negative amphicheirality).
\par
In this paper, Section\;\ref{s1} contains necessary definitions and examples of $WRP$ calculations, in Section\;\ref{s2} we give proof of $WRP$ invariance. In Section\;\ref{s3} of this paper, we show computationally generated tables of knots and links with their $WRP$ invariant, as well as a comparison to other invariants on examples.

\section{Definitions}\label{s1}

A \emph{reduced diagram} $D$ of a (non-trivial) knot or link $K$ is a generic projection of the knot or link in the plane, a regular $4$-valent graph, with extra information at each vertex indicating which arc of the link passes over the other, we assume moreover, that $D$ has no nugatory crossings, the number of crossings cannot be reduced by an immediate Reidemeister I move, and it cannot be separated by any circle in the sphere.
\par 
By a \emph{region} of a reduced diagram $D$ we mean a face of the underlying projection of the link. The diagram can be checkerboard-shaded so that every edge separates a shaded region (black region) from an unshaded one (white region). 

\subsection{Graphs}

We consider two associated with $D$ graphs (also known as Tait graphs) the shaded checkerboard graph $G_B$ and the unshaded checkerboard graph $G_W$. The graph $G_B$ has vertices corresponding to the black region regions of $D$, and an edge between two vertices for every crossing at which the corresponding regions meet. The dual graph $G_W$ is defined by analogy but the vertices correspond to the white regions.
\par
On each edge $e$ of graph $G_B$ and $G_W$ we assign its weight $weight(e)$ as a variable $w$ when the corresponding crossing was positive and variable $r$ if it was negative crossing.
\par 
An unoriented graph $G_B^*$ is then obtained from an unoriented graph $G_B$ by consolidating all multiple edges between the same pair of vertices to a single edge that has its weight equal to the product of weights all edges being consolidated. The same goes with defining graph $G_W^*$ from $G_W$.
\par
An oriented graph $G_B'$ is then obtained from an unoriented graph $G_B^*$ by doubling and weighting (by the same weight) its edges and putting opposite orientations to the resulting edges. The same goes with defining graph $G_W'$ from $G_W^*$.

\subsection{Invariant}

We define now a function\\ $WRP:\text{Reduced non-split link diagrams}\to\mathbb{Z}[w, r]$, as an unordered set:

$$WRP(D)=\left\{\sum_{C_B}\prod_{e}{weight(e)},\;\;\sum_{C_W}\prod_{e}{weight(e)}\right\},$$

where the sum is taken over all oriented cycles $C_B$ (resp $C_W$) in a graph $G_B'$ of the diagram $D$ (resp. $G_W'$) and the product is taken over all edges in a given cycle.

\begin{theorem}\label{twA}
	For any pair $D_1$ and $D_2$ of reduced alternating diagrams of a given alternating link, we have $$WRP(D_1)=WRP(D_2).$$
\end{theorem}

Therefore, for any non-trivial alternating knot or link $K$, we define $WRP(K)$ as $WRP(D)$ of any of reduced diagram $D$ representing the knot of link $K$.

\begin{remark}
	Similar ideas were considered in \cite{DiaPha21} but there was used different graphs in their constructions and the invariants defined there (combined) distinguish "all alternating knots up to $9$ crossings" (as it is stated there).
	\par
	The idea of using a pair of oppositely oriented weighted edges came to use after reading \cite{SilWil19}, in that paper they are counting trees instead of cycles, using a special diagram instead of an alternating diagram, and they recover the well-known Alexander polynomial from their construction.
\end{remark}

\subsection{Calculations}

For the (right) trefoil knot $K3a1$ we have calculations show in Figure \ref{rys1n}, that is $WRP(K3a1)=\{w^{6},\;2w^3+3w^2\}$.

\begin{figure}[h!t]
	\begin{center}
\includegraphics[width=14.5cm]{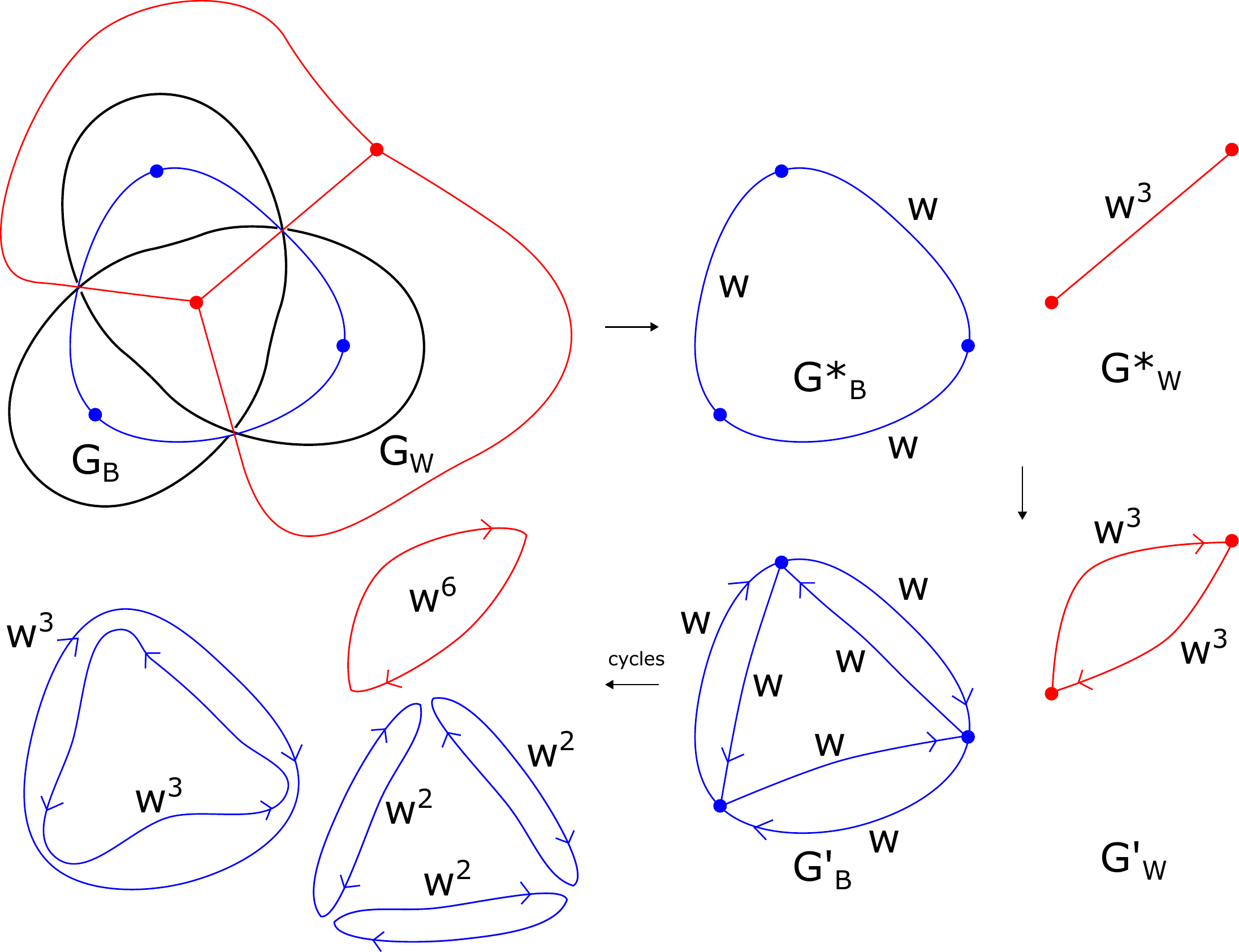}
		\caption{Calculating $WRP$ of a trefoil.\label{rys1n}}
	\end{center}
\end{figure}

We can deduce from this a generalized formula for all (positive) torus $T(2, k)$ links (for $k\geq 2$) as follows.
$$WRP(T(2,k))=\{w^{2k},\;2w^k+kw^2\}, \;\;WRP(mirror(T(2,k)))=\{r^{2k},\;2r^k+kr^2\}.$$

It can also be easily derived for all twist knots $T_{k}$ as follows.
$$WRP(T_k)=\{2w^k+w^4+(k-2)w^2,\; w^{2k-4}+2w^k+2w^2\},$$  $$WRP(mirror(T_k))=\{2r^k+r^4+(k-2)r^2,\; r^{2k-4}+2r^k+2r^2\}.$$

\section{Invariance}\label{s2}

\begin{proof}[Proof of Theorem\;\ref{twA}]

The well-known Tait's Flype Conjecture states that if $D_1$ and $D_2$ are reduced alternating diagrams, then $D_2$ can be obtained from $D_1$ by a sequence of flypes (see Figure \ref{rys2n}) if and only if $D_1$ and $D_2$ represent the same link. This conjecture has been proven true by Menasco and Thistlethwaite \cite{MenThi93}. 

\begin{figure}[h!t]
	\begin{center}
		\includegraphics[width=8cm]{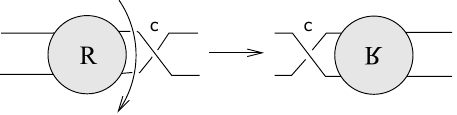}
		\caption{A flype on the tangle $R$ moves the crossing $c$ from one side of $R$ to the other, while turning $R$ by $180^\circ$ around the horizontal axis.\label{rys2n}}
	\end{center}
\end{figure}

To prove $WRP$ invariants it is sufficient now to show that the flype operation does not change $WRP$ value. First, note that there is a one-to-one correspondence between crossings before and after the move, and each corresponding crossing preserves its sign.
\par
Let us fix the reduced alternating diagram before the flype as $D_1$ and after the flype as $D_2$. We can now fix the graph $G_1$ as $G'_W$ for a diagram $D_1$ (the proof for $G_B$ will proceed analogously) and the graph $G_2$ as $G'_W$ for a diagram $D_2$.
\par
It sufficient now to prove that $$\sum_{C_1}\prod_{e}{weight(e)} = \sum_{C_2}\prod_{e}{weight(e)},$$ where $C_1$ are cycles from $G_1$ and $C_2$ are cycles from $G_2$.
\par
The general situation for $G'_W$ has to be one of two cases shown in Figures \ref{rys3}--\ref{rys4} ($D_1$ is on the left $D_2$ is on the right), where two white regions are distinguished and correspond to vertices $v_1$ and $v_2$. 
\par
We can easily conclude that if any cycle from $G_1$ does not go through both $v_1$ and $v_2$ then there is a one-to-one correspondence between cycles in $G_1$ and cycles in $G_2$ and each corresponding cycle preserves its number of edges and their weights. It is because the cycles are either outside the shown flype tangle (so they do not change) or they are completely inside tangle $R$ (plus eventually touching the unlabeled dot in the picture) so they just reverse cycle orientation and this is not considered in the final sum.
\par 
When a cycle in $G_1$ passes through both  $v_1$ and $v_2$ then there is a corresponding cycle in $G_2$ that in Case I: preserves its number, its order, and their weight of edges involved in the cycle, in Case II: for all edges between $v_1$ and $v_2$ we could pick the other twin edge with the opposite orientation, such that the resulting cycle is coherently oriented, the weights are preserved but the order may vary so this does not change the product od weights in the cycle.

\begin{figure}[h!t]
	\begin{center}
		\includegraphics[width=13cm]{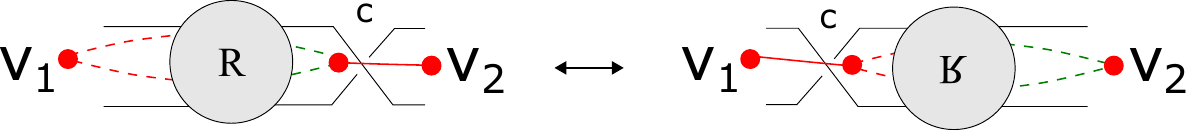}
		\caption{Case I.\label{rys4}}
	\end{center}
	\end{figure}

\begin{figure}[h!t]
	\begin{center}
		\includegraphics[width=7cm]{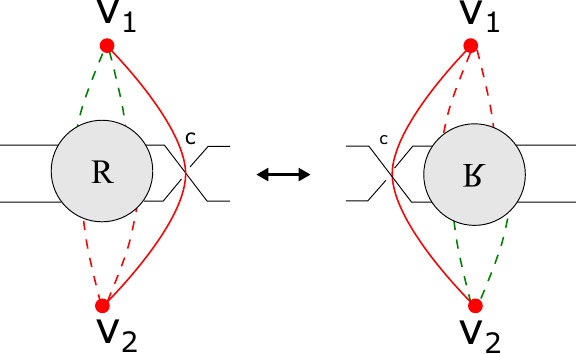}
		\caption{Case II.\label{rys3}}
	\end{center}
\end{figure}

\end{proof}


\section{Tabulation}\label{s3}

In this section, we computationally generate the following Table\;\ref{table1} of the $WRP$ invariant of prime knots up to the crossing number equal to $10$.

\par
We consider knots (and their names) that may be different in other tables up to the mirror image. Invariant $WRP$ of a given knot and $WRP$ of its mirror image differ by exchanging variables $w$ and $r$. As our computations show $WRP$ invariant is strong enough to distinguish all knots in the prime knot tables with up to $10$ crossings including their mirror images (in the sense of the negative amphicheirality). It is not a complete invariant, by the similar argument as in the Proof of Theorem\;\ref{twA}, $WRP$ cannot distinguish between a mutant pair, like knots $K11a19$ and $K11a25$ for example.
\par 
The values of the invariant for knots in the knot tables up to $13$ crossings can be found in the \texttt{tableWRP13.txt} file in the \texttt{LaTeX} source file of this paper's \texttt{arXiv} version.

\subsection{Comparisons to other knot invariants}
\ \\

We provide examples to compare the strength of $WRP$ with other knot invariants. For instance, we know that the alternating knots $K10a28$ and $K10a61$ have the same HOMFLY-PT polynomial $P$ (see \cite{HOMFLY85}, \cite{PrzTra88}) and the same signature (hence the same Khovanov and Knot Floer homologies), but $WRP(K10a28)\not=WRP(K10a61)$. In addition, we find that\\ $WRP(K10a10)\not=WRP(\text{mirror}(K10a10))$, and knots $K10a10$ and its mirror knot are indistinguishable by HOMFLY-PT polynomial, signature, and hyperbolic volume (\cite{LivMoo23}).
\par
For the two-variable Kauffman polynomial $F(a, z)$ (see \cite{Kau90}) we have, for example, $F(K11a30)=F(\text{mirror}(K11a189))$, but\\  $WRP(K11a30)\not=WRP(\text{mirror}(K11a189))$.\\ 

The same knot pair are example for the three-variable Kauffman bracket $T_L(A, B, d)$ (see \cite{Kau89}) we have, for example, $T_L(K11a30)=T_L(\text{mirror}(K11a189))=$\\
\ \\
$149A^5B^6 + 293A^6B^5d + 216A^4B^7d +
274A^7B^4d^2 + 279A^5B^6d^2 + 142A^3B^8d^2 + 154A^8B^3d^3 +165A^6B^5d^3 + 111A^4B^7d^3 + 53A^2B^9d^3 + 54A^9B^2d^4 +
56A^7B^4d^4 + 34A^5B^6d^4 + 23A^3B^8d^4 + 11AB^{10}d^4 +
11A^{10}Bd^5 + 11A^8B^3d^5 + 4A^6B^5d^5 + \\3A^4B^7d^5 +
2A^2B^9d^5 + B^{11}d^5 + A^{11}d^6 + A^9B^2d^6$, but
\ \\

$WRP(K11a30)\not=WRP(\text{mirror}(K11a189))$.\\
$WRP(K11a30) = \{2r^4w^4 + 2r^3w^4 + 2r^4w^2 + 6r^2w^4 + 2r^3w^2 + 2rw^4 + r^4 + 2r^2w^2 + 2w^4 + 2r^3 + 2rw^2 + 3r^2 + 6w^2,
2r^5w^3 + 4r^3w^4 + 2r^5w + 4r^3w^2 + 4r^2w^3 + r^4 + 7w^4 + 2r^2w + 3r^2 + 4w^2\}$ and\\ $WRP(mirror(K11a189)) = \{4r^3w^4 + 2r^2w^5 + 2r^3w^3 + 4r^3w^2 + 4r^2w^3 + r^4 + 2r^3w + 2r^2w^2 + 3w^4 + 2r^2w + 3r^2 + 4w^2,
2r^3w^4 + 2r^2w^5 + 4r^4w^2 + 4r^3w^3 + 2rw^5 + 2r^5 + 4r^2w^3 + w^4 + 4rw^2 + 5r^2 + 4w^2\}$

\ \\

\par
For the case where the other invariants are stronger, we have\\ $F(K11a75)\not=F(K11a102)$ and $P(K11a75)\not=P(K11a102)$,\\ but $WRP(K11a75)=WRP(K11a102)=\{2r^4w^4 + 2r^2w^6 + 2r^2w^4 + w^6 + 2r^4 + 4w^2, 2r^2w^6 + 2r^4w^3 + 2w^7 + 2r^2w^4 + w^6 + 2r^2w^3 + 2r^2w + 4r^2 + 4w^2\}$.\\ The diagrams of these (non-mutant) knots are shown in Figure \ref{rys5}.
\ \\

\begin{figure}[h!t]
	\begin{center}
		\includegraphics[width=9.7cm]{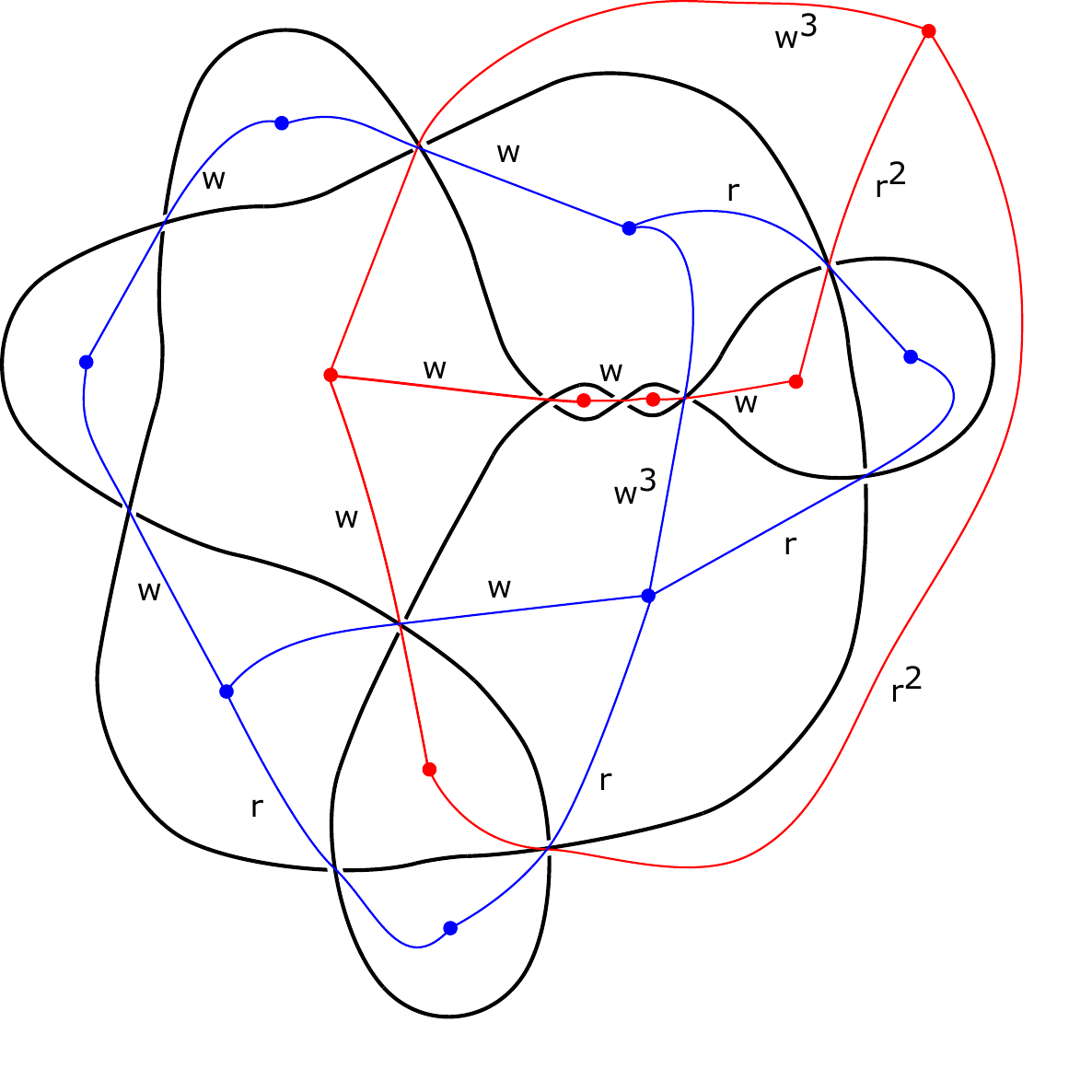}
		\includegraphics[width=9.7cm]{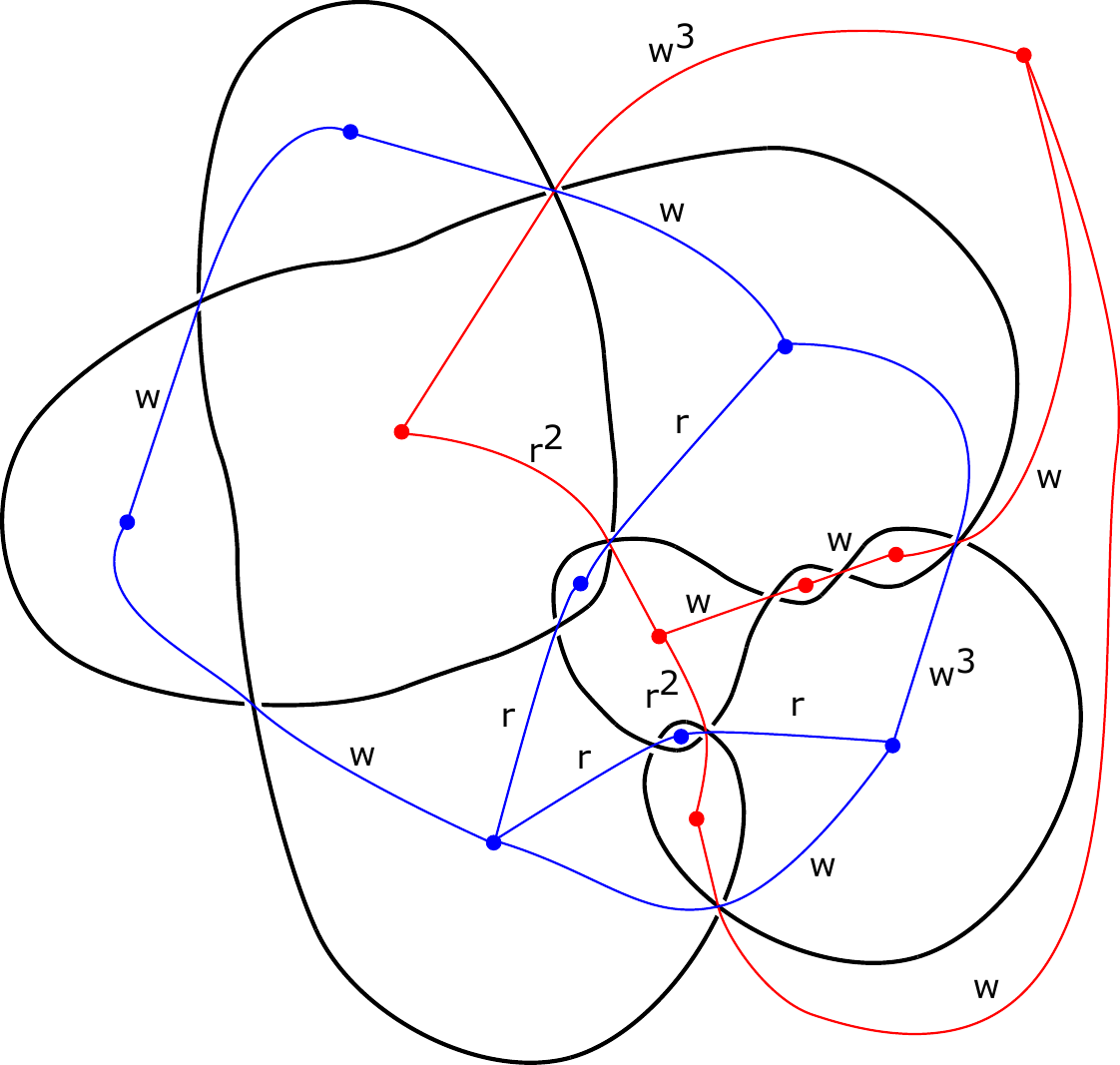}
		\caption{Knots $K11a75$ and $K11a102$ with their graphs $G_B'$ and $G_W'$.\label{rys5}}
	\end{center}
\end{figure}
\par 
As Table \ref{table2} shows $WRP$ can be considered better invariant on average. In this table we have considered all prime alternating knots up to a crossing number of $13$, without mirror images, in $6729$ total, we checked the uniqueness of an invariant of a given knot from any other knot from this set and any mirror image of any other knot in the set.

	\begin{center}
		\renewcommand{\arraystretch}{1.25}
		\begin{scriptsize}
		\begin{table}[ht]
			\caption{Occurrence of knots that are not distinguished by invariants out of 6729 total.\label{table2}}

		\begin{tabular}{l|c|r}
			
			invariant	&  knots with non-unique value & \%\\
			\hline\hline
			signature&6728&99.99\%\\
			\hline
			Alexander polynomial&4143&61.57\%\\
			\hline
			Jones polynomial&3025&44.95\% \\
			\hline
			HOMFLY-PT polynomial&1646&24.46\%\\
			\hline
			Kauffman $3$-variable bracket&810&12.04\% \\
			\hline
			Kauffman $2$-variable polynomial&755&11.22\% \\
			\hline
			{\bf WRP}&{\bf 705}&{\bf 10.48\%} 
			\\
			\hline

		\end{tabular}
			\end{table}
		\end{scriptsize}
	\end{center}

\section*{Acknowledgements}

This research was funded in whole or in part by NCN 2023/07/X/ST1/00157. For the purpose of Open Access, the author has applied a CC-BY public copyright licence to any Author Accepted Manuscript (AAM) version arising from this submission.

\newpage

\vspace*{0.2cm}
\begin{scriptsize}
	\renewcommand{\arraystretch}{1.25}
	\begin{center}
		\begin{longtable}[ht]{r|r||l}
			\caption{Knots and their $WRP$ polynomials. \label{table1}}\\
			DT-name & R-name 	&  $WRP$ polynomials\\
			\hline
			\endfirsthead
			\multicolumn{3}{c}
			{\tablename\ \thetable\ -- \textit{Continued from previous page}} \\
			DT-name & R-name 	& $WRP$ polynomials\\
			\hline
			\endhead
			\hline \multicolumn{3}{r}{\textit{Continued on next page}} \\
			\endfoot
			\hline
			\endlastfoot

$K3a1$& $3_{1}$ & $\{w^6, 2w^3 + 3w^2\}$\\

$K4a1$& $4_{1}$ & $\{r^4 + 2r^2w^2 + 2w^2, 2r^2w^2 + w^4 + 2r^2\}$\\

$K5a1$& $5_{2}$ & $\{2w^5 + w^4 + 3w^2, w^6 + 2w^5 + 2w^2\}$\\

$K5a2$& $5_{1}$ & $\{w^{10}, 2w^5 + 5w^2\}$\\

$K6a1$& $6_{3}$ & $\{2r^2w^3 + 2rw^3 + w^4 + 2r^3 + 3r^2 + w^2, 2r^3w^2 + r^4 + 2r^3w + 2w^3 + r^2 + 3w^2\}$\\

$K6a2$& $6_{2}$ & $\{2r^2w^4 + w^6 + r^4 + w^2, 2r^2w^3 + 2w^4 + 2r^2w + 2r^2 + 4w^2\}$\\

$K6a3$& $6_{1}$ & $\{2r^2w^4 + r^4 + 4w^2, w^8 + 2r^2w^4 + 2r^2\}$\\

$K7a1$& $7_{7}$ & $\{2r^4w^2 + 2r^4 + 4r^2w^2 + 3w^2, 2r^4w + 4r^2w^2 + 4r^2w + 2w^3 + 4r^2 + 3w^2\}$\\

$K7a2$& $7_{6}$ & $\{2r^2w^3 + 2w^5 + r^4 + 2r^2w^2 + w^4 + 3w^2, 2r^2w^4 + 2r^2w^3 + w^4 + 2w^3 + 2r^2 + 3w^2\}$\\

$K7a3$& $7_{5}$ & $\{2w^7 + w^6 + w^4 + 2w^2, 4w^5 + 3w^4 + 5w^2\}$\\

$K7a4$& $7_{2}$ & $\{2w^7 + w^4 + 5w^2, w^{10} + 2w^7 + 2w^2\}$\\

$K7a5$& $7_{3}$ & $\{w^8 + 2w^7 + 3w^2, 2w^7 + w^6 + 4w^2\}$\\

$K7a6$& $7_{4}$ & $\{2w^6 + 4w^4 + 7w^2, 2w^7 + 2w^6 + w^2\}$\\

$K7a7$& $7_{1}$ & $\{w^{14}, 2w^7 + 7w^2\}$\\

$K8a1$& $8_{14}$ & $\{2r^2w^5 + 2w^5 + r^4 + 2r^2w^2 + w^4 + 4w^2,$\\&  &$ 4r^2w^3 + 7w^4 + 2r^2w + 2r^2 + 4w^2\}$\\

$K8a2$& $8_{15}$ & $\{2w^6 + 4w^5 + 2w^4 + 4w^2,$\\&  &$ 2w^6 + 4w^5 + 3w^4 + 4w^3 + 6w^2\}$\\

$K8a3$& $8_{10}$ & $\{2r^3w^3 + 2r^5 + 2r^2w^3 + w^4 + 5r^2 + w^2,$\\&  &$ 2r^5w^2 + r^6 + 2r^5w + r^4 + 2w^3 + 3w^2\}$\\

$K8a4$& $8_{8}$ & $\{r^6 + 2r^3w^3 + 2r^5 + 2r^2w^3 + w^4 + 2r^2 + w^2,$\\&  &$ 2r^5w^2 + 2r^5w + r^4 + 2w^3 + 3r^2 + 3w^2\}$\\

$K8a5$& $8_{12}$ & $\{2r^4w^2 + 2r^2w^4 + 2r^4 + 2r^2w^2 + 4w^2,$\\&  &$ 2r^4w^2 + 2r^2w^4 + 2r^2w^2 + 2w^4 + 4r^2\}$\\

$K8a6$& $8_{7}$ & $\{2r^4w^3 + 2r^5 + 2rw^3 + w^4 + 5r^2 + w^2,$\\&  &$ r^8 + 2r^5w^2 + 2r^5w + 2w^3 + r^2 + 3w^2\}$\\

$K8a7$& $8_{13}$ & $\{2r^4w^3 + r^6 + 2r^5 + 2rw^3 + w^4 + 2r^2 + w^2,$\\&  &$ 2r^4w^2 + 2r^4w + 2r^4 + 2r^2w^2 + 2r^2w + 2w^3 + 5r^2 + 3w^2\}$\\

$K8a8$& $8_{2}$ & $\{w^{10} + 2r^2w^6 + r^4 + w^2,$\\&  &$ 2r^2w^5 + 2w^6 + 2r^2w + 2r^2 + 6w^2\}$\\

$K8a9$& $8_{11}$ & $\{2r^2w^4 + 2r^2w^3 + 2w^5 + r^4 + w^4 + 4w^2,$\\&  &$ 2r^2w^5 + 3w^6 + 2r^2w + 2r^2 + 3w^2\}$\\

$K8a10$& $8_{6}$ & $\{2r^2w^6 + w^6 + r^4 + 3w^2,$\\&  &$ 3w^6 + 4r^2w^3 + 2r^2 + 3w^2\}$\\

$K8a11$& $8_{1}$ & $\{2r^2w^6 + r^4 + 6w^2,$\\&  &$ w^{12} + 2r^2w^6 + 2r^2\}$\\

$K8a12$& $8_{18}$ & $\{8r^3w^2 + 2r^4 + 8r^2w^2 + 8rw^2 + 4r^2 + 4w^2,$\\&  &$ 8r^2w^3 + 8r^2w^2 + 2w^4 + 8r^2w + 4r^2 + 4w^2\}$\\

$K8a13$& $8_{5}$ & $\{2r^2w^6 + 2w^6 + r^4,$\\&  &$ 2w^6 + 4r^2w^3 + 2r^2 + 6w^2\}$\\

$K8a14$& $8_{17}$ & $\{2r^3w^3 + 2r^3w^2 + 4r^2w^3 + 2r^4 + 2rw^3 + w^4 + 2rw^2 + 4r^2 + 2w^2,$\\&  &$ 2r^3w^3 + 4r^3w^2 + 2r^2w^3 + r^4 + 2r^3w + 2w^4 + 2r^2w + 2r^2 + 4w^2\}$\\

$K8a15$& $8_{16}$ & $\{2r^4w^2 + 2r^5 + 4r^3w^2 + 4r^2w^2 + 2rw^2 + 5r^2 + 3w^2,$\\&  &$ 2r^4w^2 + 2r^4w + 4r^3w^2 + 2r^4 + 4r^3w + 2w^3 + r^2 + 3w^2\}$\\

$K8a16$& $8_{9}$ & $\{2r^3w^4 + w^6 + 2rw^4 + 2r^4 + 4r^2 + w^2,$\\&  &$ 2r^4w^3 + r^6 + 2r^4w + 2w^4 + r^2 + 4w^2\}$\\

$K8a17$& $8_{4}$ & $\{2r^3w^4 + 2rw^4 + 2r^4 + 4r^2 + 4w^2,$\\&  &$ 2r^4w^4 + w^8 + r^6 + r^2\}$\\

$K8a18$& $8_{3}$ & $\{r^8 + 2r^4w^4 + 4w^2,$\\&  &$ 2r^4w^4 + w^8 + 4r^2\}$\\

$K9a1$& $9_{30}$ & $\{2r^4w^3 + 2r^3w^3 + 2r^2w^3 + r^4 + 2r^2w^2 + 2rw^3 + w^4 + 2r^3 + 3r^2 + 2w^2,$\\&  &$ 2r^5w + 4r^3w^2 + 2r^2w^3 + r^4 + 2w^4 + 2r^2w + 3r^2 + 4w^2\}$\\

$K9a2$& $9_{22}$ & $\{2r^4w^4 + 2r^2w^4 + w^6 + 2r^4 + 2r^2w^2 + 2w^2,$\\&  &$ 2r^2w^4 + 2r^4w + 2r^2w^3 + 2w^5 + 2r^2w^2 + 2r^2w + 4r^2 + 5w^2\}$\\

$K9a3$& $9_{19}$ & $\{2r^4w^4 + 2r^2w^4 + 2r^4 + 2r^2w^2 + 5w^2,$\\&  &$ 2r^2w^4 + w^6 + 2r^4w + 2r^2w^3 + 2w^5 + 2r^2w^2 + 2r^2w + 4r^2 + 2w^2\}$\\

$K9a4$& $9_{25}$ & $\{2r^2w^5 + 2r^2w^4 + 2w^5 + r^4 + w^4 + 5w^2,$\\&  &$ 2r^2w^4 + 2w^6 + 2r^2w^3 + 2w^5 + 2r^2w^2 + 2w^4 + 2w^3 + 2r^2 + 3w^2\}$\\

$K9a5$& $9_{28}$ & $\{2w^6 + 4r^2w^3 + 4rw^3 + 2w^4 + 2r^3 + 3r^2 + 2w^2,$\\&  &$ 2r^3w^4 + 4r^3w^3 + 2r^3w^2 + r^4 + 4w^3 + r^2 + 6w^2\}$\\

$K9a6$& $9_{32}$ & $\{2r^5w^2 + 4r^4w^2 + 4r^3w^2 + 3r^4 + 2rw^2 + 4r^2 + 3w^2,$\\&  &$ 2r^5w + 2r^5 + 2r^4w + 4r^3w^2 + r^4 + 4r^3w + 4r^2w^2 + 4r^2w + 2w^3 + 4r^2 + 3w^2\}$\\

$K9a7$& $9_{24}$ & $\{2r^3w^3 + 2r^3w^2 + 2w^5 + 2r^4 + 2rw^3 + w^4 + 2rw^2 + 4r^2 + 3w^2,$\\&  &$ 2r^4w^4 + 2r^4w^3 + r^6 + w^4 + 2w^3 + r^2 + 3w^2\}$\\

$K9a8$& $9_{8}$ & $\{r^8 + 2r^4w^3 + 2r^4w^2 + 2w^5 + w^4 + 3w^2,$\\&  &$ 2r^4w^4 + 2r^4w^3 + w^4 + 2w^3 + 4r^2 + 3w^2\}$\\

$K9a9$& $9_{36}$ & $\{2r^5w^2 + 2r^4w^2 + 2r^5 + r^4 + 5r^2 + 2w^2,$\\&  &$ 2r^7 + 2r^5w^2 + r^6 + r^4 + 2r^2w^2 + w^4 + 2r^2\}$\\

$K9a10$& $9_{15}$ & $\{2r^5w^2 + r^6 + 2r^4w^2 + 2r^5 + r^4 + 2r^2 + 2w^2,$\\&  &$ 2r^7 + 2r^5w^2 + r^4 + 2r^2w^2 + w^4 + 5r^2\}$\\

$K9a11$& $9_{33}$ & $\{2r^4w^3 + 2r^4w^2 + 2r^3w^3 + 2r^3w^2 + 2r^2w^3 + r^4 + 2rw^3 + w^4 + 2r^3 + 3r^2 + 2w^2,$\\&  &$ 2r^4w^2 + 2r^4w + 6r^2w^3 + 2r^4 + 4r^2w^2 + 2w^4 + 6r^2w + 5r^2 + 4w^2\}$\\

$K9a12$& $9_{27}$ & $\{2r^2w^5 + 2r^3w^3 + 2r^3w^2 + 2r^4 + 2rw^3 + w^4 + 2rw^2 + 4r^2 + 3w^2,$\\&  &$ 2r^2w^4 + 4r^3w^2 + 2r^2w^3 + r^4 + 2r^3w + w^4 + 2w^3 + 2r^2 + 3w^2\}$\\

$K9a13$& $9_{31}$ & $\{2rw^6 + 4r^2w^3 + 4rw^3 + 2w^4 + 2r^3 + 3r^2 + 2w^2,$\\&  &$ 2r^2w^4 + 4r^2w^3 + 6r^2w^2 + 4r^2w + 4w^3 + 3r^2 + 6w^2\}$\\

$K9a14$& $9_{17}$ & $\{2r^4w^2 + 4r^2w^4 + w^6 + 2r^4 + 2w^2,$\\&  &$ 2r^4w^3 + 4r^2w^4 + 2w^5 + 4r^2w + 4r^2 + 5w^2\}$\\

$K9a15$& $9_{26}$ & $\{2r^5w^2 + 4r^3w^2 + 3r^4 + 2r^2w^2 + 2rw^2 + 4r^2 + 3w^2,$\\&  &$ 2r^6w + r^6 + 2r^4w^2 + 2r^4w + 2r^2w^2 + 2r^2w + 2w^3 + 3r^2 + 3w^2\}$\\

$K9a16$& $9_{23}$ & $\{2w^8 + 4w^5 + 2w^4 + 5w^2,$\\&  &$ 8w^5 + 6w^4 + 5w^2\}$\\

$K9a17$& $9_{14}$ & $\{r^8 + 2r^6w^2 + 2r^4w^2 + r^4 + 2r^2w^2 + 3w^2,$\\&  &$ 2r^6w + 2r^4w^2 + 2r^4w + 2r^2w^2 + 2r^2w + 2w^3 + 6r^2 + 3w^2\}$\\

$K9a18$& $9_{37}$ & $\{2r^4w^2 + 4r^2w^4 + 2r^4 + 5w^2,$\\&  &$ 2r^4w^3 + 4r^2w^4 + w^6 + 2w^5 + 4r^2w + 4r^2 + 2w^2\}$\\

$K9a19$& $9_{20}$ & $\{2w^7 + 2r^2w^4 + w^6 + 2r^2w^3 + r^4 + w^4 + 2w^2,$\\&  &$ 2r^2w^5 + 2r^2w^4 + 2r^2w^3 + 2r^2w^2 + 2w^4 + 2w^3 + 2r^2 + 7w^2\}$\\

$K9a20$& $9_{11}$ & $\{2r^6w^2 + 2r^5 + 2r^3w^2 + r^4 + 5r^2 + 2w^2,$\\&  &$ r^8 + 2r^7 + 2r^5w^2 + 2r^2w^2 + w^4 + 3r^2\}$\\

$K9a21$& $9_{21}$ & $\{2r^6w^2 + r^6 + 2r^5 + 2r^3w^2 + r^4 + 2r^2 + 2w^2,$\\&  &$ 2r^6 + 2r^4w^2 + 4r^4 + 4r^2w^2 + w^4 + 7r^2\}$\\

$K9a22$& $9_{12}$ & $\{2w^7 + 2r^2w^4 + 2r^2w^3 + r^4 + w^4 + 5w^2,$\\&  &$ 2r^2w^6 + w^8 + 2r^2w^5 + 2w^3 + 2r^2 + 3w^2\}$\\

$K9a23$& $9_{6}$ & $\{w^{10} + 2w^9 + w^4 + 2w^2,$\\&  &$ 4w^7 + 3w^4 + 7w^2\}$\\

$K9a24$& $9_{18}$ & $\{2w^7 + 2w^6 + 2w^5 + 2w^4 + 5w^2,$\\&  &$ 4w^7 + w^6 + 3w^4 + 4w^2\}$\\

$K9a25$& $9_{16}$ & $\{2w^9 + 2w^6 + w^4 + w^2,$\\&  &$ 2w^6 + 4w^5 + 4w^4 + 2w^3 + 9w^2\}$\\

$K9a26$& $9_{7}$ & $\{2w^9 + w^6 + w^4 + 4w^2,$\\&  &$ w^8 + 2w^7 + 2w^6 + 2w^5 + 5w^2\}$\\

$K9a27$& $9_{2}$ & $\{2w^9 + w^4 + 7w^2,$\\&  &$ w^{14} + 2w^9 + 2w^2\}$\\

$K9a28$& $9_{34}$ & $\{6r^4w^2 + 6r^3w^2 + 2rw^4 + r^4 + 6r^2w^2 + 2r^3 + 4rw^2 + 3r^2 + 4w^2,$\\&  &$ 4r^4w^2 + 4r^4w + 4r^2w^3 + 8r^2w^2 + 2w^4 + 4r^2w + 5r^2 + 4w^2\}$\\

$K9a29$& $9_{41}$ & $\{2r^6 + 6r^4w^2 + 3r^4 + 6r^2w^2 + 3w^2,$\\&  &$ 6r^4w^2 + 6r^4w + 2w^3 + 6r^2 + 3w^2\}$\\

$K9a30$& $9_{38}$ & $\{2w^7 + 2w^6 + 6w^5 + 5w^4 + 7w^2,$\\&  &$ 2w^7 + 2w^6 + 6w^5 + 6w^4 + 5w^2\}$\\

$K9a31$& $9_{29}$ & $\{2r^3w^4 + 2r^2w^4 + 2r^5 + 4r^3w^2 + 4r^2w^2 + 5r^2 + 4w^2,$\\&  &$ 2r^3w^4 + 2r^2w^4 + 2r^5 + 4r^3w^2 + r^4 + 4r^2w^2 + 2w^4 + 3r^2\}$\\

$K9a32$& $9_{39}$ & $\{2r^6 + 4r^4w^2 + 2r^5 + 4r^3w^2 + 2r^4 + 2r^3 + 3r^2 + 2w^2,$\\&  &$ 2r^6 + 6r^4w^2 + 4r^4 + 2r^2w^2 + w^4 + 7r^2\}$\\

$K9a33$& $9_{9}$ & $\{2w^9 + w^8 + w^6 + 2w^2,$\\&  &$ 2w^7 + 2w^6 + 2w^5 + w^4 + 7w^2\}$\\

$K9a34$& $9_{13}$ & $\{2w^8 + 3w^6 + 2w^4 + 6w^2,$\\&  &$ 2w^7 + 3w^6 + 2w^5 + w^4 + 4w^2\}$\\

$K9a35$& $9_{4}$ & $\{2w^9 + w^8 + 5w^2,$\\&  &$ w^{10} + 2w^9 + 4w^2\}$\\

$K9a36$& $9_{5}$ & $\{2w^8 + 2w^6 + 2w^4 + 9w^2,$\\&  &$ w^{10} + 2w^9 + w^6 + w^2\}$\\

$K9a37$& $9_{40}$ & $\{12rw^4 + 12r^2w^2 + 6w^4 + 2r^3 + 12rw^2 + 3r^2 + 6w^2,$\\&  &$ 6r^2w^4 + 12r^2w^3 + 6r^2w^2 + 4w^3 + 3r^2 + 6w^2\}$\\

$K9a38$& $9_{3}$ & $\{w^{12} + 2w^9 + 3w^2,$\\&  &$ 2w^9 + w^6 + 6w^2\}$\\

$K9a39$& $9_{10}$ & $\{7w^6 + 6w^2,$\\&  &$ 2w^9 + 2w^6 + 3w^2\}$\\

$K9a40$& $9_{35}$ & $\{6w^6 + 9w^2,$\\&  &$ 2w^9 + 3w^6\}$\\

$K9a41$& $9_{1}$ & $\{w^{18},$\\&  &$ 2w^9 + 9w^2\}$\\

$K10a1$& $10_{60}$ & $\{2r^6w^2 + 4r^4w^2 + 3r^4 + 6r^2w^2 + 4w^2,$\\&  &$ 2r^4w^2 + 4r^4w + 4r^2w^3 + 4r^2w^2 + 2w^4 + 4r^2w + 6r^2 + 4w^2\}$\\

$K10a2$& $10_{59}$ & $\{2r^4w^3 + 2r^2w^5 + 2r^2w^3 + 2w^5 + 2r^4 + 4r^2w^2 + w^4 + 4w^2,$\\&  &$ 2r^4w^3 + 4r^2w^4 + 2r^2w^3 + 3w^4 + 2r^2w + 4r^2 + 4w^2\}$\\

$K10a3$& $10_{73}$ & $\{2r^6w^2 + 2r^5w^2 + 2r^4w^2 + 2r^3w^2 + 2r^4 + 2r^2w^2 + 2r^3 + 3r^2 + 3w^2,$\\&  &$ 2r^5w + 2r^5 + 2r^4w + 2r^3w^2 + r^4 + 2r^3w + 4r^2w^2 + 4r^2w + 2w^3 + 5r^2 + 3w^2\}$\\

$K10a4$& $10_{72}$ & $\{2r^2w^7 + 2w^7 + w^6 + r^4 + 2r^2w^2 + w^4 + 3w^2,$\\&  &$ 2r^2w^4 + 2r^2w^3 + 4w^5 + 2r^2w^2 + 4w^4 + 2r^2w + 4w^3 + 2r^2 + 8w^2\}$\\

$K10a5$& $10_{36}$ & $\{2r^2w^7 + 2w^7 + r^4 + 2r^2w^2 + w^4 + 6w^2,$\\&  &$ w^8 + 2r^2w^5 + 4w^6 + 2r^2w^3 + 2w^4 + 2r^2w + 2r^2 + 4w^2\}$\\

$K10a6$& $10_{57}$ & $\{4r^3w^3 + 4r^5 + 2r^2w^3 + 3r^4 + w^4 + 5r^2 + w^2,$\\&  &$ 2r^6w^2 + 2r^6w + 2r^5 + 2r^3w^2 + 2r^4 + 2r^3w + 2w^3 + 3r^2 + 3w^2\}$\\

$K10a7$& $10_{81}$ & $\{2r^4w^3 + 2r^4w^2 + 2r^3w^3 + 2r^3w^2 + 2w^5 + r^4 + w^4 + 2r^3 + 3r^2 + 3w^2,$\\&  &$ 2r^3w^4 + 2r^3w^3 + 2r^2w^4 + 2r^5 + 2r^2w^3 + r^4 + w^4 + 2w^3 + 3r^2 + 3w^2\}$\\

$K10a8$& $10_{80}$ & $\{2w^8 + 2w^7 + w^6 + 2w^5 + 2w^4 + 3w^2,$\\&  &$ 2w^7 + 4w^6 + 4w^5 + w^4 + 2w^3 + 8w^2\}$\\

$K10a9$& $10_{55}$ & $\{2w^8 + 2w^7 + 2w^5 + 2w^4 + 6w^2,$\\&  &$ 2w^7 + 5w^6 + 4w^5 + w^4 + 2w^3 + 5w^2\}$\\

$K10a10$& $10_{71}$ & $\{2r^2w^5 + 2r^4w^2 + 2r^3w^2 + 2r^2w^3 + r^4 + 2rw^3 + w^4 + 2r^3 + 3r^2 + 3w^2,$\\&  &$ 2r^5w^2 + 2r^3w^4 + 2r^5w + 2r^3w^3 + r^4 + 2r^2w^2 + w^4 + 2w^3 + 3r^2 + 3w^2\}$\\

$K10a11$& $10_{88}$ & $\{2r^3w^4 + 4r^4w^2 + 2r^2w^4 + 4r^3w^2 + 2rw^4 + r^4 + 4r^2w^2 + 2r^3 + 4rw^2 + 3r^2 + 5w^2,$\\&  &$ 2r^4w^3 + 2r^4w^2 + 4r^2w^4 + 2r^4w + 4r^2w^3 + 4r^2w^2 + w^4 + 4r^2w + 2w^3 + 5r^2 + 3w^2\}$\\

$K10a12$& $10_{97}$ & $\{2r^2w^6 + 6r^2w^4 + 2w^6 + r^4 + 4w^4 + 8w^2,$\\&  &$ 2r^2w^5 + 4w^6 + 6r^2w^3 + 12w^4 + 2r^2w + 2r^2 + 4w^2\}$\\

$K10a13$& $10_{49}$ & $\{3w^8 + 2w^7 + 2w^5 + w^4 + 4w^2,$\\&  &$ 2w^8 + 2w^7 + 4w^5 + 3w^4 + 2w^3 + 8w^2\}$\\

$K10a14$& $10_{53}$ & $\{2w^7 + 2w^6 + 4w^5 + 5w^4 + 8w^2,$\\&  &$ 2w^8 + 2w^7 + w^6 + 4w^5 + 3w^4 + 2w^3 + 5w^2\}$\\

$K10a15$& $10_{47}$ & $\{2r^5w^3 + 2r^7 + 2r^2w^3 + w^4 + 7r^2 + w^2,$\\&  &$ r^{10} + 2r^7w^2 + 2r^7w + r^4 + 2w^3 + 3w^2\}$\\

$K10a16$& $10_{51}$ & $\{2r^5w^3 + 2r^7 + r^6 + 2r^2w^3 + w^4 + 4r^2 + w^2,$\\&  &$ 2r^5w^2 + 2r^5w + 2r^4w^2 + 2r^5 + 2r^4w + 2r^4 + 2w^3 + 3r^2 + 3w^2\}$\\

$K10a17$& $10_{78}$ & $\{2w^6 + 4r^2w^3 + 4w^5 + r^4 + 2r^2w^2 + 2w^4 + 4w^2,$\\&  &$ 2r^2w^6 + 4r^2w^5 + 2r^2w^4 + w^4 + 4w^3 + 2r^2 + 6w^2\}$\\

$K10a18$& $10_{77}$ & $\{2r^3w^3 + 4r^5 + 4r^2w^3 + 3r^4 + w^4 + 5r^2 + w^2,$\\&  &$ 2r^7w^2 + 2r^7w + r^6 + r^4 + 2w^3 + 2r^2 + 3w^2\}$\\

$K10a19$& $10_{34}$ & $\{r^{10} + 2r^5w^3 + 2r^7 + 2r^2w^3 + w^4 + 2r^2 + w^2,$\\&  &$ 2r^7w^2 + 2r^7w + r^4 + 2w^3 + 5r^2 + 3w^2\}$\\

$K10a20$& $10_{58}$ & $\{2r^4w^4 + 4r^2w^4 + 2r^4 + 6w^2,$\\&  &$ 2r^4w^2 + 4r^2w^4 + 2w^6 + 4r^2w^2 + 3w^4 + 4r^2\}$\\

$K10a21$& $10_{89}$ & $\{2r^6w^2 + 2r^5w^2 + 4r^4w^2 + 4r^3w^2 + 2r^4 + 2r^3 + 3r^2 + 3w^2,$\\&  &$ 2r^6 + 10r^4w + 4r^4 + 10r^2w^2 + 10r^2w + 2w^3 + 7r^2 + 3w^2\}$\\

$K10a22$& $10_{70}$ & $\{2r^5w^2 + 2r^2w^4 + 4r^3w^2 + 3r^4 + 2rw^2 + 4r^2 + 4w^2,$\\&  &$ 2r^6w^2 + 2r^4w^4 + r^6 + 2r^2w^2 + 2w^4 + 3r^2\}$\\

$K10a23$& $10_{35}$ & $\{r^8 + 2r^6w^2 + 2r^4w^2 + 2r^2w^4 + r^4 + 4w^2,$\\&  &$ 2r^6w^2 + 2r^4w^4 + 2r^2w^2 + 2w^4 + 6r^2\}$\\

$K10a24$& $10_{96}$ & $\{2r^6w^2 + 6r^4w^2 + 2r^2w^4 + 3r^4 + 4r^2w^2 + 4w^2,$\\&  &$ 2r^6w + 4r^4w^2 + 4r^4w + 2r^2w^3 + 8r^2w^2 + 2w^4 + 2r^2w + 6r^2 + 4w^2\}$\\

$K10a25$& $10_{45}$ & $\{2r^3w^4 + 2r^4w^2 + 2r^3w^2 + 2rw^4 + r^4 + 6r^2w^2 + 2r^3 + 4rw^2 + 3r^2 + 5w^2,$\\&  &$ 2r^4w^3 + 2r^2w^4 + 2r^4w + 2r^2w^3 + 6r^2w^2 + w^4 + 4r^2w + 2w^3 + 5r^2 + 3w^2\}$\\

$K10a26$& $10_{39}$ & $\{2r^2w^5 + 2w^7 + 2r^2w^4 + w^6 + r^4 + w^4 + 3w^2,$\\&  &$ 4r^2w^5 + 4w^6 + 3w^4 + 2r^2w + 2r^2 + 6w^2\}$\\

$K10a27$& $10_{75}$ & $\{6r^4w^2 + 3r^4 + 6r^2w^2 + 4w^2,$\\&  &$ 2r^6w + 6r^4w^2 + 6r^2w^3 + 2w^4 + 6r^2w + 6r^2 + 4w^2\}$\\

$K10a28$& $10_{56}$ & $\{2r^2w^7 + 2r^2w^4 + w^6 + 2w^5 + r^4 + w^4 + 3w^2,$\\&  &$ 4w^6 + 6r^2w^3 + 3w^4 + 2r^2 + 6w^2\}$\\

$K10a29$& $10_{38}$ & $\{2r^2w^7 + 2r^2w^4 + 2w^5 + r^4 + w^4 + 6w^2,$\\&  &$ 5w^6 + 6r^2w^3 + 3w^4 + 2r^2 + 3w^2\}$\\

$K10a30$& $10_{40}$ & $\{4r^4w^3 + 4r^5 + 3r^4 + 2rw^3 + w^4 + 5r^2 + w^2,$\\&  &$ 2r^5w^2 + 2r^6 + 2r^5w + 2r^3w^2 + 2r^4 + 2r^3w + 2w^3 + 3r^2 + 3w^2\}$\\

$K10a31$& $10_{42}$ & $\{2r^3w^5 + 2r^4w^2 + 2r^3w^2 + 2r^2w^3 + r^4 + 2rw^3 + w^4 + 2r^3 + 3r^2 + 3w^2,$\\&  &$ 2r^4w^2 + 2r^2w^4 + 2r^4w + 2r^2w^3 + 2r^4 + 6r^2w^2 + w^4 + 2r^2w + 2w^3 + 5r^2 + 3w^2\}$\\

$K10a32$& $10_{44}$ & $\{2r^2w^5 + 2r^4w^2 + 2r^2w^3 + 2w^5 + 2r^4 + 4r^2w^2 + w^4 + 4w^2,$\\&  &$ 2r^4w^3 + 2r^4w^2 + 2r^2w^4 + 4r^2w^3 + 4r^2w^2 + 2r^2w + 4w^3 + 4r^2 + 6w^2\}$\\

$K10a33$& $10_{14}$ & $\{2r^2w^7 + w^8 + 2w^7 + r^4 + 2r^2w^2 + 4w^2,$\\&  &$ 2r^2w^5 + 4w^6 + 2r^2w^3 + 3w^4 + 2r^2w + 2r^2 + 6w^2\}$\\

$K10a34$& $10_{30}$ & $\{2r^2w^6 + 2r^2w^4 + 2w^6 + r^4 + 2r^2w^2 + 4w^4 + 8w^2,$\\&  &$ 2r^2w^5 + 5w^6 + 2r^2w^3 + 3w^4 + 2r^2w + 2r^2 + 3w^2\}$\\

$K10a35$& $10_{41}$ & $\{2r^4w^2 + 2r^2w^4 + 4r^2w^3 + 2w^5 + 2r^4 + 2r^2w^2 + w^4 + 4w^2,$\\&  &$ 2r^4w^4 + 2r^2w^5 + 4r^2w^3 + 3w^4 + 2r^2w + 4r^2 + 4w^2\}$\\

$K10a36$& $10_{113}$ & $\{6r^2w^5 + 6rw^5 + 2rw^4 + 2w^5 + 4r^2w^2 + w^4 + 2r^3 + 4rw^2 + 3r^2 + 5w^2,$\\&  &$ 4r^2w^4 + 12r^2w^3 + 8r^2w^2 + 6w^4 + 4r^2w + 6w^3 + 3r^2 + 7w^2\}$\\

$K10a37$& $10_{67}$ & $\{2r^2w^5 + 2w^7 + 2r^2w^4 + r^4 + w^4 + 6w^2,$\\&  &$ 4r^2w^5 + 5w^6 + 3w^4 + 2r^2w + 2r^2 + 3w^2\}$\\

$K10a38$& $10_{69}$ & $\{6r^4w^2 + 4r^3w^2 + 2r^4 + 2r^3 + 3r^2 + 3w^2,$\\&  &$ 2r^7 + 4r^5w + 2r^4w + 2r^3w^2 + r^4 + 2r^3w + 4r^2w^2 + 4r^2w + 2w^3 + 5r^2 + 3w^2\}$\\

$K10a39$& $10_{87}$ & $\{2r^3w^5 + 4r^2w^5 + 2rw^5 + 2r^3w^2 + 2r^4 + w^4 + 2rw^2 + 4r^2 + 4w^2,$\\&  &$ 4r^3w^3 + 6r^3w^2 + 4r^2w^3 + r^4 + 2r^3w + 7w^4 + 2r^2w + 2r^2 + 4w^2\}$\\

$K10a40$& $10_{66}$ & $\{4w^7 + 3w^6 + 2w^4 + 3w^2,$\\&  &$ 2w^7 + 4w^6 + 4w^5 + 6w^4 + 6w^3 + 10w^2\}$\\

$K10a41$& $10_{62}$ & $\{2r^7 + 2r^4w^3 + 2r^3w^3 + w^4 + 7r^2 + w^2,$\\&  &$ 2r^7w^2 + r^8 + 2r^7w + r^6 + 2w^3 + 3w^2\}$\\

$K10a42$& $10_{65}$ & $\{2r^7 + 2r^4w^3 + r^6 + 2r^3w^3 + w^4 + 4r^2 + w^2,$\\&  &$ 2r^6w^2 + 2r^6w + r^6 + 2r^4w^2 + 2r^4w + 2r^4 + 2w^3 + 4r^2 + 3w^2\}$\\

$K10a43$& $10_{12}$ & $\{2r^7 + 2r^4w^3 + r^6 + 2r^3w^3 + w^4 + 4r^2 + w^2,$\\&  &$ 2r^7w^2 + r^8 + 2r^7w + 2w^3 + 3r^2 + 3w^2\}$\\

$K10a44$& $10_{28}$ & $\{2r^7 + 2r^4w^3 + 2r^6 + 2r^3w^3 + w^4 + r^2 + w^2,$\\&  &$ 2r^6w^2 + 2r^6w + 2r^4w^2 + 2r^4w + 2r^4 + 2w^3 + 7r^2 + 3w^2\}$\\

$K10a45$& $10_{101}$ & $\{6w^7 + 2w^6 + 2w^5 + 5w^4 + 8w^2,$\\&  &$ 6w^6 + 10w^5 + 6w^4 + 4w^3 + 6w^2\}$\\

$K10a46$& $10_{92}$ & $\{2r^2w^6 + 2r^2w^5 + 2w^6 + 2r^2w^3 + 4w^5 + r^4 + 2r^2w^2 + 2w^4 + 4w^2,$\\&  &$ 2r^2w^5 + 4r^2w^4 + 4r^2w^3 + 4w^5 + 2r^2w^2 + 4w^4 + 4w^3 + 2r^2 + 8w^2\}$\\

$K10a47$& $10_{95}$ & $\{2r^5w^3 + 2r^4w^3 + 2r^3w^3 + 4r^5 + 2r^2w^3 + 3r^4 + w^4 + 5r^2 + w^2,$\\&  &$ 2r^5w^2 + 2r^5w + 2r^4w^2 + 4r^5 + 2r^4w + 2r^3w^2 + 3r^4 + 2r^3w + 2r^2w^2 + 2r^2w + 2w^3 + 5r^2 + 3w^2\}$\\

$K10a48$& $10_{54}$ & $\{2r^3w^5 + 2r^2w^5 + 2r^5 + w^4 + 5r^2 + 3w^2,$\\&  &$ 2r^5w^3 + 2r^5w^2 + r^6 + w^6 + 2w^5 + r^4 + 2w^2\}$\\

$K10a49$& $10_{37}$ & $\{2r^3w^5 + 2r^2w^5 + r^6 + 2r^5 + w^4 + 2r^2 + 3w^2,$\\&  &$ 2r^5w^3 + 2r^5w^2 + w^6 + 2w^5 + r^4 + 3r^2 + 2w^2\}$\\

$K10a50$& $10_{84}$ & $\{2r^2w^6 + 2r^2w^5 + 2rw^6 + 2rw^5 + 2r^2w^3 + 2rw^3 + 2w^4 + 2r^3 + 3r^2 + 3w^2,$\\&  &$ 4r^2w^4 + 6r^2w^3 + 4w^5 + 6r^2w^2 + 3w^4 + 2r^2w + 3r^2 + 5w^2\}$\\

$K10a51$& $10_{63}$ & $\{4w^7 + 2w^6 + 2w^4 + 6w^2,$\\&  &$ 3w^8 + 4w^7 + 2w^6 + 4w^3 + 6w^2\}$\\

$K10a52$& $10_{43}$ & $\{2r^2w^5 + 2r^4w^2 + 2rw^5 + 2r^3w^2 + 2r^2w^3 + r^4 + w^4 + 2r^3 + 3r^2 + 3w^2,$\\&  &$ 2r^5w^2 + 2r^5w + 2r^2w^4 + 2r^3w^2 + 2r^2w^3 + r^4 + w^4 + 2w^3 + 3r^2 + 3w^2\}$\\

$K10a53$& $10_{29}$ & $\{2r^2w^6 + 2r^4w^2 + 2r^2w^4 + w^6 + 2r^4 + 3w^2,$\\&  &$ 2r^4w^3 + 2r^2w^5 + 2r^4w + 2r^2w^3 + 2r^2w^2 + 3w^4 + 4r^2 + 4w^2\}$\\

$K10a54$& $10_{13}$ & $\{2r^2w^6 + 2r^4w^2 + 2r^2w^4 + 2r^4 + 6w^2,$\\&  &$ 2r^4w^4 + 2r^2w^6 + w^8 + 2r^2w^2 + w^4 + 4r^2\}$\\

$K10a55$& $10_{32}$ & $\{2r^3w^5 + 2rw^5 + 2r^3w^2 + 2w^5 + 2r^4 + w^4 + 2rw^2 + 4r^2 + 4w^2,$\\&  &$ 4r^4w^3 + r^6 + 2r^4w + 7w^4 + r^2 + 4w^2\}$\\

$K10a56$& $10_{5}$ & $\{2r^6w^3 + 2r^7 + 2rw^3 + w^4 + 7r^2 + w^2,$\\&  &$ r^{12} + 2r^7w^2 + 2r^7w + 2w^3 + r^2 + 3w^2\}$\\

$K10a57$& $10_{23}$ & $\{2r^6w^3 + 2r^7 + r^6 + 2rw^3 + w^4 + 4r^2 + w^2,$\\&  &$ 3r^6 + 4r^4w^2 + 4r^4w + 2w^3 + 4r^2 + 3w^2\}$\\

$K10a58$& $10_{27}$ & $\{2r^4w^3 + 2r^3w^3 + 4r^5 + 3r^4 + 2rw^3 + w^4 + 5r^2 + w^2,$\\&  &$ 2r^6w^2 + 2r^6w + 3r^6 + 2r^2w^2 + 2r^2w + 2w^3 + 4r^2 + 3w^2\}$\\

$K10a59$& $10_{2}$ & $\{w^{14} + 2r^2w^8 + r^4 + w^2,$\\&  &$ 2r^2w^7 + 2w^8 + 2r^2w + 2r^2 + 8w^2\}$\\

$K10a60$& $10_{21}$ & $\{w^8 + 2r^2w^5 + 2w^7 + 2r^2w^4 + r^4 + 4w^2,$\\&  &$ 2r^2w^7 + 2w^8 + w^6 + 2r^2w + 2r^2 + 5w^2\}$\\

$K10a61$& $10_{25}$ & $\{2r^2w^6 + 2w^7 + w^6 + 2r^2w^3 + r^4 + w^4 + 3w^2,$\\&  &$ 2r^2w^5 + 2r^2w^4 + 2w^6 + 4w^5 + w^4 + 2r^2w + 2r^2 + 6w^2\}$\\

$K10a62$& $10_{74}$ & $\{4r^2w^4 + 2w^6 + r^4 + 2r^2w^2 + 4w^4 + 8w^2,$\\&  &$ 2r^2w^7 + 2w^8 + 2w^6 + 2r^2w + 2r^2 + 2w^2\}$\\

$K10a63$& $10_{18}$ & $\{2r^4w^5 + r^8 + 2r^4w^2 + 2w^5 + w^4 + 4w^2,$\\&  &$ 4r^4w^3 + 2r^4w + 7w^4 + 4r^2 + 4w^2\}$\\

$K10a64$& $10_{10}$ & $\{r^{10} + 2r^6w^3 + 2r^7 + 2rw^3 + w^4 + 2r^2 + w^2,$\\&  &$ 2r^6w^2 + 2r^6w + 2r^6 + 2r^2w^2 + 2r^2w + 2w^3 + 7r^2 + 3w^2\}$\\

$K10a65$& $10_{7}$ & $\{2r^2w^6 + 2w^7 + 2r^2w^3 + r^4 + w^4 + 6w^2,$\\&  &$ w^{10} + 2r^2w^7 + 2w^8 + 2r^2w + 2r^2 + 3w^2\}$\\

$K10a66$& $10_{107}$ & $\{2r^4w^3 + 2r^2w^5 + 2r^4w^2 + 2r^3w^3 + 2rw^5 + 2r^3w^2 + r^4 + w^4 + 2r^3 + 3r^2 + 3w^2,$\\&  &$ 4r^4w^2 + 4r^2w^4 + 2r^4w + 4r^2w^3 + 2r^4 + 4r^2w^2 + w^4 + 2r^2w + 2w^3 + 5r^2 + 3w^2\}$\\

$K10a67$& $10_{68}$ & $\{2r^6w^3 + 2r^7 + 2r^6 + 2rw^3 + w^4 + r^2 + w^2,$\\&  &$ 2r^6 + 4r^4w^2 + 4r^4w + 2w^3 + 7r^2 + 3w^2\}$\\

$K10a68$& $10_{15}$ & $\{2r^4w^5 + 2rw^5 + 2r^5 + w^4 + 5r^2 + 3w^2,$\\&  &$ r^8 + 2r^5w^3 + 2r^5w^2 + w^6 + 2w^5 + r^2 + 2w^2\}$\\

$K10a69$& $10_{31}$ & $\{2r^4w^5 + r^6 + 2rw^5 + 2r^5 + w^4 + 2r^2 + 3w^2,$\\&  &$ 2r^4w^3 + 2r^4w^2 + w^6 + 2r^2w^3 + 2w^5 + 2r^4 + 2r^2w^2 + 5r^2 + 2w^2\}$\\

$K10a70$& $10_{6}$ & $\{2r^2w^8 + w^{10} + r^4 + 3w^2,$\\&  &$ 2w^8 + 2r^2w^5 + w^6 + 2r^2w^3 + 2r^2 + 5w^2\}$\\

$K10a71$& $10_{24}$ & $\{2r^2w^6 + 2r^2w^5 + 2w^5 + r^4 + w^4 + 6w^2,$\\&  &$ 2w^8 + 2r^2w^5 + 2w^6 + 2r^2w^3 + 2r^2 + 2w^2\}$\\

$K10a72$& $10_{105}$ & $\{2r^4w^3 + 4r^2w^5 + 2r^4w^2 + 2w^5 + 2r^4 + 4r^2w^2 + w^4 + 4w^2,$\\&  &$ 2r^4w^2 + 4r^2w^4 + 2r^4w + 8r^2w^3 + 4r^2w^2 + 4w^3 + 4r^2 + 6w^2\}$\\

$K10a73$& $10_{76}$ & $\{2r^2w^8 + 2w^6 + r^4 + 2w^2,$\\&  &$ 2w^6 + 4r^2w^3 + 4w^5 + 2r^2w^2 + w^4 + 2r^2 + 6w^2\}$\\

$K10a74$& $10_{20}$ & $\{2r^2w^8 + w^6 + r^4 + 5w^2,$\\&  &$ w^{10} + 2w^8 + 2r^2w^5 + 2r^2w^3 + 2r^2 + 3w^2\}$\\

$K10a75$& $10_{1}$ & $\{2r^2w^8 + r^4 + 8w^2,$\\&  &$ w^{16} + 2r^2w^8 + 2r^2\}$\\

$K10a76$& $10_{112}$ & $\{6r^5w^2 + 2r^6 + 4r^4w^2 + 4r^3w^2 + 4r^2w^2 + 6rw^2 + 6r^2 + 4w^2,$\\&  &$ 4r^4w^3 + r^6 + 4r^4w^2 + 4r^4w + 4r^2w^3 + 4r^2w^2 + 2w^4 + 4r^2w + 3r^2 + 4w^2\}$\\

$K10a77$& $10_{114}$ & $\{6r^5w^2 + 3r^6 + 4r^4w^2 + 4r^3w^2 + 4r^2w^2 + 6rw^2 + 3r^2 + 4w^2,$\\&  &$ 4r^4w^3 + 4r^4w^2 + 4r^4w + 4r^2w^3 + 4r^2w^2 + 2w^4 + 4r^2w + 6r^2 + 4w^2\}$\\

$K10a78$& $10_{79}$ & $\{2r^3w^5 + 2r^2w^5 + w^6 + 2r^5 + w^4 + 5r^2,$\\&  &$ 2r^5w^3 + 2r^5w^2 + r^6 + 2w^5 + r^4 + 5w^2\}$\\

$K10a79$& $10_{48}$ & $\{2r^4w^5 + 2rw^5 + w^6 + 2r^5 + w^4 + 5r^2,$\\&  &$ r^8 + 2r^5w^3 + 2r^5w^2 + 2w^5 + r^2 + 5w^2\}$\\

$K10a80$& $10_{52}$ & $\{2r^4w^5 + r^6 + 2rw^5 + w^6 + 2r^5 + w^4 + 2r^2,$\\&  &$ 2r^4w^3 + 2r^4w^2 + 2r^2w^3 + 2w^5 + 2r^4 + 2r^2w^2 + 5r^2 + 5w^2\}$\\

$K10a81$& $10_{46}$ & $\{2r^2w^8 + w^{10} + w^6 + r^4,$\\&  &$ 2w^8 + 2r^2w^5 + 2r^2w^3 + 2r^2 + 8w^2\}$\\

$K10a82$& $10_{50}$ & $\{2r^2w^6 + 2r^2w^5 + w^6 + 2w^5 + r^4 + w^4 + 3w^2,$\\&  &$ 2w^8 + 2r^2w^5 + w^6 + 2r^2w^3 + 2r^2 + 5w^2\}$\\

$K10a83$& $10_{82}$ & $\{2r^3w^5 + w^8 + 4r^2w^5 + 2rw^5 + 2r^3w^2 + 2r^4 + 2rw^2 + 4r^2 + 2w^2,$\\&  &$ 2r^3w^5 + 2r^3w^4 + 2r^2w^5 + 2w^6 + 2r^3w^2 + r^4 + 2r^3w + 2r^2w + 2r^2 + 6w^2\}$\\

$K10a84$& $10_{83}$ & $\{2r^3w^4 + 4r^2w^4 + 4r^3w^2 + 2rw^4 + 2r^4 + 4r^2w^2 + 2w^4 + 4rw^2 + 4r^2 + 6w^2,$\\&  &$ 2r^3w^5 + 2r^3w^4 + 2r^2w^5 + 3w^6 + 2r^3w^2 + r^4 + 2r^3w + 2r^2w + 2r^2 + 3w^2\}$\\

$K10a85$& $10_{119}$ & $\{4r^4w^3 + 4r^3w^3 + 4r^3w^2 + 4r^2w^3 + 7r^4 + 2rw^3 + w^4 + 2rw^2 + 4r^2 + 2w^2,$\\&  &$ 4r^4w^3 + 4r^4w^2 + 6r^4w + 2r^2w^3 + 2r^4 + 4r^2w^2 + 2w^4 + 2r^2w + 6r^2 + 4w^2\}$\\

$K10a86$& $10_{85}$ & $\{2r^6w^2 + 2r^7 + 2r^5w^2 + 2r^4w^2 + 2r^3w^2 + 2r^2w^2 + 2rw^2 + 7r^2 + 3w^2,$\\&  &$ r^8 + 2r^6w^2 + 2r^6w + 2r^5w^2 + 2r^5w + 2r^3w^2 + r^4 + 2r^3w + 2w^3 + r^2 + 3w^2\}$\\

$K10a87$& $10_{86}$ & $\{2r^6w^2 + 2r^7 + 2r^5w^2 + r^6 + 2r^4w^2 + 2r^3w^2 + 2r^2w^2 + 2rw^2 + 4r^2 + 3w^2,$\\&  &$ 2r^5w^2 + 2r^5w + 2r^4w^2 + 2r^4w + 4r^3w^2 + 3r^4 + 4r^3w + 2r^2w^2 + 2r^2w + 2w^3 + 5r^2 + 3w^2\}$\\

$K10a88$& $10_{118}$ & $\{4r^4w^3 + 2r^4w^2 + 2r^3w^3 + 2r^5 + 4r^3w^2 + 2r^2w^3 + 4r^2w^2 + 4rw^3 + w^4 + 2rw^2 + 5r^2 + 3w^2,$\\&  &$ 4r^3w^4 + 2r^3w^3 + 2r^2w^4 + 2r^3w^2 + 4r^2w^3 + 2w^5 + r^4 + 4r^3w + 4r^2w^2 + 2r^2w + 3r^2 + 5w^2\}$\\

$K10a89$& $10_{122}$ & $\{4r^4w^2 + 16r^3w^2 + 7r^4 + 10r^2w^2 + 8rw^2 + 4r^2 + 4w^2,$\\&  &$ 4r^4w^3 + 10r^4w^2 + 4r^4w + 4r^2w^3 + 2w^4 + 4r^2w + 6r^2 + 4w^2\}$\\

$K10a90$& $10_{121}$ & $\{8r^2w^4 + 12rw^4 + 8r^2w^2 + 6w^4 + 2r^3 + 8rw^2 + 3r^2 + 7w^2,$\\&  &$ 4r^2w^5 + 8r^2w^4 + 4r^2w^3 + 4w^5 + 4r^2w^2 + 3w^4 + 2r^2w + 3r^2 + 5w^2\}$\\

$K10a91$& $10_{94}$ & $\{2r^3w^5 + 2r^3w^4 + 2rw^5 + w^6 + 4r^2w^3 + 2rw^4 + 2r^4 + w^4 + 4r^2 + w^2,$\\&  &$ 2r^3w^5 + 2r^3w^4 + 2w^6 + 2r^3w^2 + 4r^2w^3 + r^4 + 2r^3w + 2r^2 + 6w^2\}$\\

$K10a92$& $10_{90}$ & $\{2r^3w^5 + 2r^3w^4 + 2rw^5 + 4r^2w^3 + 2rw^4 + 2r^4 + w^4 + 4r^2 + 4w^2,$\\&  &$ 2r^3w^5 + 2r^3w^4 + 3w^6 + 2r^3w^2 + 4r^2w^3 + r^4 + 2r^3w + 2r^2 + 3w^2\}$\\

$K10a93$& $10_{109}$ & $\{2r^4w^3 + 2r^3w^4 + 2r^3w^3 + 2r^2w^4 + 2r^5 + 2r^2w^3 + 2rw^3 + 2w^4 + 5r^2 + w^2,$\\&  &$ 2r^4w^3 + 2r^3w^4 + 2r^4w^2 + 2r^3w^3 + 2r^3w^2 + 2w^5 + 2r^4 + 2r^3w + r^2 + 5w^2\}$\\

$K10a94$& $10_{115}$ & $\{2r^3w^4 + 6r^4w^2 + 2r^2w^4 + 6r^3w^2 + 2rw^4 + r^4 + 2r^2w^2 + 2w^4 + 2r^3 + 2rw^2 + 3r^2 + 5w^2,$\\&  &$ 2r^4w^3 + 2r^4w^2 + 6r^2w^4 + 2r^4w + 6r^2w^3 + 2r^4 + 2r^2w^2 + w^4 + 2r^2w + 2w^3 + 5r^2 + 3w^2\}$\\

$K10a95$& $10_{106}$ & $\{2r^3w^4 + 4r^2w^5 + 2r^3w^3 + w^6 + 2rw^4 + 2r^4 + 2rw^3 + w^4 + 4r^2 + w^2,$\\&  &$ 2r^3w^4 + 2r^2w^5 + 4r^3w^3 + 2w^6 + 2r^3w^2 + r^4 + 2r^2w + 2r^2 + 6w^2\}$\\

$K10a96$& $10_{98}$ & $\{2r^2w^6 + 2r^2w^4 + 2w^6 + 4r^2w^3 + 4w^5 + r^4 + 2w^4 + 4w^2,$\\&  &$ 2r^2w^6 + 2r^2w^4 + 2w^6 + 4r^2w^3 + 4w^5 + w^4 + 2r^2 + 6w^2\}$\\

$K10a97$& $10_{102}$ & $\{2r^3w^4 + 4r^2w^5 + 2r^3w^3 + 2rw^4 + 2r^4 + 2rw^3 + w^4 + 4r^2 + 4w^2,$\\&  &$ 2r^3w^4 + 2r^2w^5 + 4r^3w^3 + 3w^6 + 2r^3w^2 + r^4 + 2r^2w + 2r^2 + 3w^2\}$\\

$K10a98$& $10_{111}$ & $\{4r^2w^5 + 2r^2w^4 + 2w^6 + 4w^5 + r^4 + 2r^2w^2 + 2w^4 + 4w^2,$\\&  &$ 4r^2w^5 + 4w^6 + 4r^2w^3 + 3w^4 + 2r^2 + 6w^2\}$\\

$K10a99$& $10_{117}$ & $\{2r^5w^2 + 4r^4w^2 + 4r^5 + 8r^3w^2 + 3r^4 + 4r^2w^2 + 2rw^2 + 5r^2 + 3w^2,$\\&  &$ 4r^5w^2 + 6r^5w + 2r^4w^2 + 2r^4w + 2r^3w^2 + 3r^4 + 2r^3w + 2r^2w^2 + 2r^2w + 2w^3 + 5r^2 + 3w^2\}$\\

$K10a100$& $10_{110}$ & $\{2r^4w^3 + 2r^2w^5 + 2r^4w^2 + 2r^2w^4 + 2r^2w^3 + 2w^5 + 2r^4 + 2r^2w^2 + w^4 + 4w^2,$\\&  &$ 2r^4w^3 + 2r^2w^5 + 4r^2w^4 + 2r^4w + 2r^2w^3 + 3w^4 + 4r^2 + 4w^2\}$\\

$K10a101$& $10_{93}$ & $\{2r^4w^4 + 2r^3w^4 + 2r^2w^4 + 2r^5 + 2r^3w^2 + 2rw^4 + 2r^2w^2 + 5r^2 + 5w^2,$\\&  &$ 2r^4w^4 + 2r^3w^4 + 2r^3w^3 + w^6 + 2r^4w + 2r^3w^2 + 2w^5 + 2r^4 + 2r^3w + r^2 + 2w^2\}$\\

$K10a102$& $10_{120}$ & $\{12w^6 + 14w^4 + 10w^2,$\\&  &$ 4w^7 + 6w^6 + 8w^5 + 6w^4 + 4w^3 + 6w^2\}$\\

$K10a103$& $10_{99}$ & $\{2r^4w^4 + 4r^3w^3 + 2r^5 + 4r^2w^3 + 2rw^4 + 2w^4 + 5r^2 + w^2,$\\&  &$ 2r^4w^4 + 4r^3w^3 + 2r^4w + 4r^3w^2 + 2w^5 + 2r^4 + r^2 + 5w^2\}$\\

$K10a104$& $10_{100}$ & $\{2r^7 + 4r^5w^2 + 2r^4w^2 + 2r^3w^2 + 4r^2w^2 + 7r^2 + 3w^2,$\\&  &$ 4r^5w^2 + r^6 + 4r^5w + 2r^4w^2 + 2r^4w + 2r^4 + 2w^3 + 3w^2\}$\\

$K10a105$& $10_{103}$ & $\{2r^7 + 4r^5w^2 + r^6 + 2r^4w^2 + 2r^3w^2 + 4r^2w^2 + 4r^2 + 3w^2,$\\&  &$ 4r^5w^2 + 4r^5w + 2r^4w^2 + 2r^4w + 2r^4 + 2w^3 + 3r^2 + 3w^2\}$\\

$K10a106$& $10_{91}$ & $\{2r^4w^4 + 2r^4w^3 + 2r^3w^3 + 2r^5 + 2r^2w^3 + 2rw^4 + 2rw^3 + 2w^4 + 5r^2 + w^2,$\\&  &$ 2r^4w^4 + 2r^4w^3 + r^6 + 2r^4w^2 + 2r^4w + 2r^2w^3 + 2w^5 + 2r^2w^2 + 2r^2 + 5w^2\}$\\

$K10a107$& $10_{17}$ & $\{2r^4w^5 + w^8 + 2rw^5 + 2r^5 + 5r^2 + w^2,$\\&  &$ 2r^5w^4 + r^8 + 2r^5w + 2w^5 + r^2 + 5w^2\}$\\

$K10a108$& $10_{19}$ & $\{2r^4w^5 + w^8 + r^6 + 2rw^5 + 2r^5 + 2r^2 + w^2,$\\&  &$ 2r^4w^4 + 2r^2w^4 + 2r^4w + 2w^5 + 2r^4 + 2r^2w + 5r^2 + 5w^2\}$\\

$K10a109$& $10_{33}$ & $\{2r^4w^4 + r^6 + 2r^4w^2 + 2r^5 + 2rw^4 + 2w^4 + 2rw^2 + 2r^2 + 5w^2,$\\&  &$ 2r^4w^4 + 2r^2w^4 + w^6 + 2r^4w + 2w^5 + 2r^4 + 2r^2w + 5r^2 + 2w^2\}$\\

$K10a110$& $10_{9}$ & $\{w^{10} + 2r^3w^6 + 2rw^6 + 2r^4 + 4r^2 + w^2,$\\&  &$ 2r^4w^5 + r^6 + 2w^6 + 2r^4w + r^2 + 6w^2\}$\\

$K10a111$& $10_{26}$ & $\{2r^3w^4 + 2r^3w^3 + 2rw^4 + 2w^5 + 2r^4 + 2rw^3 + w^4 + 4r^2 + 4w^2,$\\&  &$ 2r^4w^5 + r^6 + 3w^6 + 2r^4w + r^2 + 3w^2\}$\\

$K10a112$& $10_{22}$ & $\{2r^3w^6 + 2rw^6 + w^6 + 2r^4 + 4r^2 + 3w^2,$\\&  &$ 4r^4w^3 + r^6 + 3w^6 + r^2 + 3w^2\}$\\

$K10a113$& $10_{4}$ & $\{2r^3w^6 + 2rw^6 + 2r^4 + 4r^2 + 6w^2,$\\&  &$ w^{12} + 2r^4w^6 + r^6 + r^2\}$\\

$K10a114$& $10_{8}$ & $\{2r^4w^6 + w^{10} + r^8 + w^2,$\\&  &$ 2r^4w^5 + 2w^6 + 2r^4w + 4r^2 + 6w^2\}$\\

$K10a115$& $10_{16}$ & $\{r^8 + 2r^4w^4 + 2r^4w^3 + 2w^5 + w^4 + 4w^2,$\\&  &$ 2r^4w^5 + 3w^6 + 2r^4w + 4r^2 + 3w^2\}$\\

$K10a116$& $10_{11}$ & $\{2r^4w^6 + r^8 + w^6 + 3w^2,$\\&  &$ 4r^4w^3 + 3w^6 + 4r^2 + 3w^2\}$\\

$K10a117$& $10_{3}$ & $\{2r^4w^6 + r^8 + 6w^2,$\\&  &$ w^{12} + 2r^4w^6 + 4r^2\}$\\

$K10a118$& $10_{104}$ & $\{4r^4w^3 + 2r^3w^4 + 2r^2w^4 + 2r^5 + 4rw^3 + 2w^4 + 5r^2 + w^2,$\\&  &$ 4r^4w^3 + r^6 + 4r^4w^2 + 2r^2w^4 + 2w^5 + 2r^2w + 2r^2 + 5w^2\}$\\

$K10a119$& $10_{108}$ & $\{4r^4w^3 + 2r^3w^4 + r^6 + 2r^2w^4 + 2r^5 + 4rw^3 + 2w^4 + 2r^2 + w^2,$\\&  &$ 4r^4w^3 + 4r^4w^2 + 2r^2w^4 + 2w^5 + 2r^2w + 5r^2 + 5w^2\}$\\

$K10a120$& $10_{116}$ & $\{4r^5w^2 + 2r^6 + 6r^4w^2 + 4r^3w^2 + 6r^2w^2 + 4rw^2 + 6r^2 + 4w^2,$\\&  &$ 2r^4w^3 + 4r^3w^3 + 2r^4w + 8r^3w^2 + 2r^2w^3 + 2r^4 + 4r^3w + 2w^4 + 2r^2w + 2r^2 + 4w^2\}$\\

$K10a121$& $10_{123}$ & $\{10r^4w^2 + 2r^5 + 10r^3w^2 + 10r^2w^2 + 10rw^2 + 5r^2 + 5w^2,$\\&  &$ 10r^2w^4 + 10r^2w^3 + 2w^5 + 10r^2w^2 + 10r^2w + 5r^2 + 5w^2\}$\\

$K10a122$& $10_{64}$ & $\{2r^3w^6 + 2rw^6 + 2w^6 + 2r^4 + 4r^2,$\\&  &$ 4r^4w^3 + r^6 + 2w^6 + r^2 + 6w^2\}$\\

$K10a123$& $10_{61}$ & $\{2r^4w^6 + r^8 + 2w^6,$\\&  &$ 4r^4w^3 + 2w^6 + 4r^2 + 6w^2\}$\\

			\hline
		\end{longtable}
	\end{center}
	
\end{scriptsize}

\end{document}